\documentclass[leqno]{aadmbook}

\usepackage{amsthm}
\usepackage{amsfonts}
\usepackage{amsmath}

\usepackage{amssymb,enumerate}
\usepackage{latexsym}

\begin{document}

\newtheorem{theorem}{Theorem}[section]
\newtheorem{lemma}[theorem]{Lemma}
\newtheorem{proposition}[theorem]{Proposition}
\newtheorem{corollary}[theorem]{Corollary}
\theoremstyle{definition}
\newtheorem{definition}[theorem]{Definition}
\newtheorem{example}{Example}[section]
\newtheorem{remark}[theorem]{Remark}


\oddsidemargin 16.5mm
\evensidemargin 16.5mm

\thispagestyle{plain}

{\noindent\small
This is a preprint of a paper whose final and definite form is published in\\ 
Applicable Analysis and Discrete Mathematics (AADM), ISSN 1452-8630 (printed),\\ 
ISSN 2406-100X (online), available online at http:/$\!$/pefmath.etf.rs.\\
Submitted to AADM 02/Nov/2014; Revised and Resubmitted to AADM 22/Jul/2015; 
Revised 10/Feb/2016; Accepted 11/March/2016. Cite this work as:\newline
Appl. Anal. Discrete Math., Vol.~10 (2016), no.~1, DOI: 10.2298/AADM160311004F.}

\vspace{3cc}


\begin{center}

{\large\bf A HUKUHARA APPROACH TO THE STUDY OF HYBRID FUZZY SYSTEMS ON TIME SCALES
\rule{0mm}{6mm}\renewcommand{\thefootnote}{}
\footnotetext{

\rule{2.4mm}{0mm}Keywords and Phrases. Fuzzy hybrid systems,
fuzzy dynamical systems, practical stability, time scales.}}

\vspace{1cc}
{\large\it Omid S. Fard, Delfim F. M. Torres, Mohadeseh R. Zadeh}

\vspace{1cc}
\parbox{24cc}{{\small
We introduce a new approach to study the practical stability
of hybrid fuzzy systems on time scales in the Lyapunov sense.
Our method is based on the delta-Hukuhara derivative
for fuzzy valued functions and allow us to obtain
new interesting stability criteria.
We also show the validity of the results of
{\sc M. Sambandham:}
{\it Hybrid fuzzy systems on time scales,}
Dynam. Systems Appl. {\bf 12} (2003), no.~1-2, 217--227,
by embedding the space of all fuzzy subsets into a suitable Banach space.}}
\end{center}

\vspace{1cc}


\section{Introduction}

In natural systems engineering, the lowest level in the hierarchical structure
is usually characterized by the dynamics of a continuous variable while the
highest level is described by a logical decision making mechanism \cite{Branicky}.
The interaction of these different levels, with their different types of information,
leads to a \emph{hybrid system}. Examples of real world hybrid systems include systems
with relays, switches, hysteresis, disk drivers, transmissions, step motors,
constrained robots, automated transportation systems, modern manufacturing
and flight control systems \cite{van}.

The mathematical modeling of dynamic processes is often discussed
in the literature via difference or differential equations. In spite of this tendency
of independence between discrete and continuous dynamic systems, there is a striking
similarity between both theories. From a modeling point of view, it is perhaps more
realistic to model a phenomenon by a dynamic system, which incorporates both
discrete and continuous times simultaneously, namely by considering time as an arbitrary
closed set of reals, called a time scale. The theory of time scales was introduced
in 1988 by Stefan Hilger and is now a well-developed subject \cite{boh,bohner}.
Here we use dynamic systems on time scales as a model to hybrid systems.
Motivation comes from the fact that hybrid systems refer
to the construction of models combining both continuous and discrete dynamics.
They are useful for describing the interaction between physical and computational
processes, such as in digital feedback control systems \cite{HM}.

In the study of stability of dynamical systems, Lyapunov functions play a central role
in the proof of stability of equilibrium points on the state space \cite{MR2783351}.
Moreover, since practical stability only requires to stabilize a system into
a region of the phase space, Lyapunov functions have been widely used in applications.
More precisely, when one intends to analyze a real world phenomenon, it is also
necessary to deal with uncertain factors. In this case, the theory of fuzzy
sets is one of the best non-statistical or non-probabilistic approach,
which leads us to investigate stability of fuzzy dynamical models.
In recent years, the use of hybrid fuzzy systems
has increased drastically. For instance, Kim and Sakthivel \cite{kim}
studied the predictor-corrector method for hybrid fuzzy differential equations,
Nieto et al. \cite{MR2561685} and Allahviranloo and Salahshour \cite{allah}
investigated the numerical solution of (hybrid) fuzzy differential equations
using a generalized Euler approximation method, and Ahmadian et al.
\cite{ahmadian} employed a numerical algorithm to solve a first-order
hybrid fuzzy differential equation based on the high-order Runge--Kutta method.
Particularly for hybrid fuzzy systems on time scales, the theory of practical
stability has developed rather intensively in the last few years -- see
\cite{MR3151691,laksh3,samban,MR2771572,MR2991148} and references therein.

In \cite{laksh3}, Lakshmikantham and Vatsala investigated practical stability
of hybrid systems on time scales. The results of \cite{laksh3} were then
extended by Sambandham for hybrid fuzzy systems on time scales \cite{samban}.
In \cite{MR2771572}, Lyapunov-like functions for hybrid dynamic equations on
time scales with the state space $\mathbb{R}^n$ are studied using Dini
derivatives and comparison principles. Related studies are found in
\cite{MR2991148}, for incremental stability of stochastic hybrid systems,
in \cite{MR3151691} for practical stability of discrete hybrid systems
with different initial times, and in \cite{slyn} for stability
analysis of abstract Takagi--Sugeno fuzzy impulsive systems.

In this paper, we propose a new approach to investigate practical stability
of hybrid fuzzy systems on time scales. Our method is more general,
being based on the use of the delta-Hukuhara derivative,
which has recently been defined for fuzzy valued functions \cite{omid}.
After a short review of this calculus of fuzzy functions on time scales
in Section~\ref{sec:prelim}, we prove new stability criteria
in Section~\ref{sec:MR}. We proceed with Section~\ref{sec:rem:pr}, where an
inconsistency in \cite{samban} is noted and fixed,
ending with Section~\ref{sec:conc} of conclusions.


\section{Preliminaries}
\label{sec:prelim}

In this section, we present some basic concepts
and results on the calculus of fuzzy functions
on time scales. We assume, however,
the reader to be familiar with the standard calculus
on time scales \cite{boh,bohner}.

\begin{definition}[Upper $\Delta$-Dini derivative \cite{laksh3}]
Let $\mathbb{T}$ be a time scale with forward jump operator $\sigma$,
$f: \mathbb{T}\rightarrow \mathbb{R}$ be a real valued function,
and $U_{\mathbb{T}}$ be a neighborhood of $t\in \mathbb{T}$.
We call $D^+_{\Delta} f(t) \in \mathbb{R}$ the upper $\Delta$-Dini derivative
(or upper $\Delta$-generalized derivative) of $f$ at $t$, provided that
for any $\varepsilon>0$ there exists a right neighborhood
$U_{\varepsilon} \subset U_{\mathbb{T}}$ of $t$ (i.e.,
$U_{\varepsilon}=(t,t+\varepsilon)\cap \mathbb{T}$) such that
$$
\frac{f(\sigma(t))- f(s)}{\mu^*(t,s)}< D^+_{\Delta} f(t) + \varepsilon
$$
for $s\in U_{\varepsilon}$ and $s>t$, where $\mu^*(t,s)=\sigma(t)-s$.
\end{definition}

\begin{proposition}[See \cite{laksh3}]
Let $\mathbb{T}$ be a time scale with forward jump operator $\sigma$
and graininess function $\mu$. If $f: \mathbb{T}\rightarrow \mathbb{R}$
is a continuous function at $t\in \mathbb{T}$ and $t$ is right-scattered, then
the upper $\Delta$-Dini derivative of $f$ coincides with
the standard delta-derivative:
$$
D^+_{\Delta}f(t)=f^\Delta (t)=\frac{f(\sigma(t))-f(t)}{\mu(t)}.
$$
\end{proposition}

\begin{definition}[Fuzzy set \cite{bede}]
Let $X$ be a nonempty  set. A fuzzy set $u$ in $X$ is characterized
by its  membership function $u:X \rightarrow [0,1]$. Then,
for each $x\in X$ we interpret $u(x)$ as the degree of membership of the element $x$
in the fuzzy set $u$: $u(x)=0$ corresponded to non membership; $0 <u(x)< 1$ to
partial membership; and $u(x) = 1$ to full membership.
\end{definition}

\begin{definition}[The space of fuzzy numbers $\mathbb{R}_\mathcal{F}$
-- see, e.g., \cite{samban}]
We denote by $\mathbb{R}_\mathcal{F}$ the class of fuzzy subsets
of the real axis $u : \mathbb{R}\to[0,1]$ satisfying the following properties:
\begin{enumerate}[(i)]
\item $u$ is normal,  i.e., there exists $x_0\in\mathbb{R}$ with $u(x_0) = 1$;

\item $u$ is a convex fuzzy set, i.e.,
$u(k x+(1-k)y)\geq\min\{u(x),u(y)\}$
for all $k\in[0,1]$ and $x,y\in\mathbb{R}$;

\item $u$ is upper semicontinuous on $\mathbb{R}$;

\item $\overline{\{x \in \mathbb{R} : u(x)> 0\}}$ is compact,
where $\overline{A}$ denotes the closure of a set $A$.
\end{enumerate}
We call $\mathbb{R}_\mathcal{F}$ the space of fuzzy numbers.
\end{definition}

Clearly, $\mathbb{R}\subset\mathbb{R}_\mathcal{F}$ (because $\mathbb{R}$ is
understood as $\mathbb{R}=\{\chi_{x} : x \mbox{ is a real number}\}$). For
later purposes, we define $\widetilde{0}=\chi_{\{0\}}$, that is, $\widetilde{0}$
is the fuzzy set defined by $\widetilde{0}(x)=1$ if $x=0$
and $\widetilde{0}(x)=0$ if $x\neq 0$.

\begin{definition}[The $\alpha$-level set \cite{bede}]
For $0<\alpha\leqslant 1$, the $\alpha$-level set
$[u]^\alpha $ of a fuzzy set $u$ on $\mathbb{R}$ is defined as
$$
[u]^{ \alpha }=\{x \in \mathbb{R} : u(x)\geq\alpha\},
$$
while its support $[u]^0$ is the closure of the union
of all the level sets, that is,
$$
[u]^0=\overline{\bigcup_{\alpha\in(0,1]}[u]^ \alpha }
=\overline{\{x \in \mathbb{R} : u(x)> 0\}}.
$$
\end{definition}

\begin{remark}
For any $\alpha \in [0,1]$, $[v]^{\alpha}$
is a bounded closed interval in $\mathbb{R}$, presented by
$[v]^{\alpha}=[\underline{v}^{\alpha},\overline{v}^{\alpha}]$,
where $\underline{v}^{\alpha},\overline{v}^{\alpha}\in\mathbb{R}$.
\end{remark}

For $u,v\in \mathbb{R}_\mathcal{F}$ and $\lambda\in \mathbb{R}$,
the sum of two fuzzy numbers and the multiplication between a
real and a fuzzy number are defined respectively by
$$
[u\oplus v]^\alpha=[u]^\alpha + [v]^\alpha
=\left\{x+y : x\in [u]^\alpha, y\in [v]^\alpha\right\}
$$
and
$$
[\lambda \cdot u]^\alpha
=\lambda [u]^\alpha=\left\{\lambda x : x \in [u]^\alpha\right\}
$$
for all $\alpha \in [0,1]$, where $[u]^\alpha + [v]^\alpha$ is the usual addition
of two intervals of $\mathbb{R}$ and $\lambda [u]^\alpha$ is the usual product
of a number and a subset of $\mathbb{R}$.

\begin{definition}[The gH-difference \cite{omid}]
Given $u,v \in  \mathbb{R}_\mathcal{F}$, the gH-difference
is the fuzzy number $w$, if it exists, such that
$$
u\ominus _{gH} v=w \Leftrightarrow
u=v\oplus w \text{ or } v=u\oplus (-1)\cdot w.
$$
\end{definition}

\begin{remark}
If $u\ominus_{gH}v$ exists, then its $\alpha$-level set is given by
$$
[u\ominus _{gH} v]^\alpha
= \left[\min\left\{\underline{u}^\alpha-\underline{v}^\alpha,
\overline{u}^\alpha-\overline{v}^\alpha\right\},
\max\left\{\underline{u}^\alpha-\underline{v}^\alpha,
\overline{u}^\alpha-\overline{v}^\alpha\right\}\right].
$$
\end{remark}

Let $A$ and $B$ be two nonempty bounded subsets of a metric space $(X,d)$.
The Hausdorff distance between $A$ and $B$ is given by
\begin{equation}
\label{eq8}
d_H(A,B)=\max\left[ \mathop{\sup}_{a \in A}\mathop{\inf}_{b \in B} d(a,b),
\mathop{\sup}_{b \in B}\mathop{\inf}_{a \in A}d(b,a) \right].
\end{equation}

\begin{definition}[Hausdorff distance between two fuzzy numbers]
\label{def1}
The Hausdorff distance between two fuzzy numbers is the function
$d_\infty :\mathbb{R}_\mathcal{F} \times \mathbb{R}_\mathcal{F}\rightarrow
\mathbb{R}_+\cup \{0\}$ defined in terms of the Hausdorff distance between
their level sets, that is,
\begin{equation*}
d_\infty (u,v)=\mathop{\sup}\left\{d_H ([u]^\alpha,[v]^\alpha)
: \alpha\in [0,1] \right\}.
\end{equation*}
\end{definition}

\begin{proposition}[See \cite{diamond}]
\label{pro1}
The pair $(\mathbb{R}_\mathcal{F},d_\infty)$ is a complete metric space
and the following properties hold:
\begin{enumerate}[(i)]
\item $d_\infty(u\oplus w,v\oplus w)=d_\infty(u,v)$
for all $u,v,w \in \mathbb{R}_\mathcal{F}$;

\item $d_\infty(k \cdot u,k \cdot v)= |k| d_\infty(u,v)$
for all $u,v \in \mathbb{R}_\mathcal{F}$ and for all $k\in \mathbb{R}$;

\item $d_\infty(u\oplus v,w\oplus e) \leq d_\infty(u,w)+ d_\infty(v,e)$
for all $u,v,w,e \in \mathbb{R}_\mathcal{F}$.
\end{enumerate}
\end{proposition}

\begin{definition}[The set of rd-continuous functions
from $\mathbb{T}$ to $\mathbb{R}_\mathcal{F}$ \cite{omid}]
A mapping $f : \mathbb{T} \rightarrow \mathbb{R}_\mathcal{F}$
is rd-continuous if it is continuous at each right-dense point
and its left-side limits exist (finite) at left-dense points in $\mathbb{T}$.
We denote the set of rd-continuous functions from
$\mathbb{T}$ to $\mathbb{R}_\mathcal{F}$ by
$C_{rd}[\mathbb{T},\mathbb{R}_\mathcal{F}]$.
\end{definition}

\begin{definition}[The $\Delta$-Hukuhara derivative of $f$ at $t$ \cite{omid}]
Assume $f : \mathbb{T} \rightarrow \mathbb{R}_\mathcal{F}$
is a fuzzy function and let $t\in \mathbb{T}^\kappa$.
Let $\Delta_H f(t)$ be an element of $\mathbb{R}_\mathcal{F}$ (provided it exists)
with the property that given any $\varepsilon>0$, there exists a neighborhood
$U_{\mathbb{T}}$ of $t$ (i.e., $U_{\mathbb{T}}=(t-\delta,t+\delta)
\cap \mathbb{T}$ for some $\delta>0$) such that
$$
d_\infty\left[f(t+h)\ominus _{gH}f(\sigma(t)),\Delta_H f(t)(h-\mu(t))\right]
\leq\varepsilon(h-\mu(t))
$$
and
$$
d_\infty\left[f(\sigma(t))\ominus _{gH}f(t-h),\Delta_H f(t)(\mu(t)+h)\right]
\leq\varepsilon(\mu(t)+h)
$$
for all $t-h,t+h \in U_{\mathbb{T}}$ with $0\leq h<\delta$.
We call $\Delta_H f(t)$ the $\Delta$-Hukuhara derivative of $f$ at $t$.
We say that $f$ is $\Delta_H$-differentiable at $t$
if its $\Delta_H$-derivative exists at $t$. Moreover,
we say that $f$ is $\Delta_H$-differentiable on $\mathbb{T}^\kappa$
if its $\Delta_H$-derivative exists at each $t\in \mathbb{T}^\kappa$.
The fuzzy function $\Delta_H f : \mathbb{T}^\kappa
\rightarrow \mathbb{R}_\mathcal{F}$ is then called the
$\Delta_H$-derivative of $f$ on $\mathbb{T}^\kappa$.
\end{definition}

The next result shows that the $\Delta_H$-derivative is well defined.

\begin{proposition}[See \cite{omid}]
If the $\Delta_H$-derivative of $f$ at
$t \in \mathbb{T}^\kappa$ exists, then it is unique.
\end{proposition}

\begin{theorem}[See \cite{omid}]
\label{th1}
Let $\mathbb{T}$ be an arbitrary time scale and consider a function
$f: \mathbb{T}\rightarrow \mathbb{R}_\mathcal{F}$. Then, the following
holds for all $t \in \mathbb{T}^\kappa$:
\begin{enumerate}[(i)]
\item If $f$ is continuous at $t$ and $t$ is right-scattered,
then $f$ is $\Delta_H$-differentiable at $t$ with
\begin{equation*}
\Delta_H f(t) =\frac{f(\sigma(t))\ominus_{gH} f(t)}{\mu(t)}.
\end{equation*}

\item If $t$ is right-dense, then $f$ is $\Delta_H$-differentiable at $t$
if and only if both limits
$$
\lim_{h\rightarrow 0^+} \frac{f(t+h)\ominus _{gH} f(t)}{h}
\quad \text{ and } \quad
\lim_{h\rightarrow 0^+} \frac{f(t)\ominus _{gH} f(t-h)}{h}
$$
exist and satisfy the equalities
$$
\lim_{h\rightarrow 0^+} \frac{f(t+h)\ominus_{gH} f(t)}{h}
=\lim_{h\rightarrow 0^+} \frac{f(t)\ominus_{gH} f(t-h)}{h}
=\Delta_H f(t).
$$
\end{enumerate}
\end{theorem}

Let us denote by $\mathbb{R}_\mathcal{F}^n$ the space of all fuzzy subsets
$u$ of $\mathbb{R}^n$ that satisfy the assumptions
\begin{enumerate}[(i)]
\item $u$ maps $\mathbb{R}^n$ onto $I=[0,1]$;
\item $u$ is fuzzy convex;
\item $u$ is upper semicontinuous;
\item $[u]^0$ is a compact subset of $\mathbb{R}^n$.
\end{enumerate}
The most commonly used metric on $\mathbb{R}_\mathcal{F}^n$ involves
the Hausdorff metric distance between the level sets of the fuzzy sets,
which is defined  for any $\mathbf{u},\mathbf{v} \in \mathbb{R}_\mathcal{F}^n$ as
$$
D_\infty(\mathbf{u},\mathbf{v})
=\mathop{\sup}\left\{d_H ([\mathbf{u}]^\alpha,[\mathbf{v}]^\alpha)
: \alpha\in [0,1] \right\},
$$
where $d_H$ is the  Hausdorff distance defined by \eqref{eq8}. The pair
$(\mathbb{R}_\mathcal{F}^n,D_\infty)$ is a complete metric space.
Proposition~\ref{pro1} can be extended to
$(\mathbb{R}_\mathcal{F}^n,D_\infty)$ \cite{laksh}.

\begin{definition}
Assume that $\mathbf{u}:\mathbb{T}\rightarrow \mathbb{R}_\mathcal{F}^n$
is a fuzzy vector valued function and $t \in \mathbb{T}^\kappa$. We say
that $\mathbf{u}$ is $\Delta_H$-differentiable at $t$, if its all components
are $\Delta_H$-differentiable at $t$.
\end{definition}

\begin{definition}
The function $f: \mathbb{T} \times \mathbb{R}_\mathcal{F}^n
\rightarrow \mathbb{R}_\mathcal{F}^n$ is rd-continuous if
\begin{enumerate}[(i)]
\item it is continuous at each $(t,\mathbf{x})$ with $t$ right-dense, and

\item the limits $\lim_{(s,\mathbf{y})\rightarrow(t^-,\mathbf{x})}
f(s,\mathbf{y})=f(t^-,\mathbf{x})$ (i.e., for all $\varepsilon >0$
there exists $\delta >0$ such that
$D_\infty (\mathbf{y},\mathbf{x})+ |s-t^-| <\delta
\Longrightarrow D_\infty (f(s,\mathbf{y}), f(t^-,\mathbf{x}))<\varepsilon$)
and $\lim_{\mathbf{y}\rightarrow \mathbf{x}} f(t,\mathbf{y})
=f(t,\mathbf{x})$ (i.e., for all $\varepsilon >0$ there exists
$\delta >0$ such that $D_\infty(\mathbf{y},\mathbf{x}) <\delta
\Longrightarrow D_\infty(f(t,\mathbf{y}), f(t,\mathbf{x}))<\varepsilon$)
exist at each $(t,\mathbf{x})$ with $t$ left-dense.
\end{enumerate}
\end{definition}

We denote by $C_{rd}[\mathbb{T} \times
\mathbb{R}_\mathcal{F}^n, \mathbb{R}_\mathcal{F}^n]$
the set of all rd-continuous functions
$f: \mathbb{T} \times  \mathbb{R}_\mathcal{F}^n
\rightarrow \mathbb{R}_\mathcal{F}^n$. The spaces
$C_{rd}[\mathbb{T}\times \mathbb{R}_\mathcal{F}^n, \mathbb{R}_+]$
and $C_{rd}[\mathbb{T} \times \mathbb{R}_+, \mathbb{R}]$
are defined similarly.


\section{Main Results}
\label{sec:MR}

We investigate the practical stability of hybrid fuzzy systems
on time scales, defined through delta-Hukuhara derivatives.
Consider the fuzzy dynamic system
\begin{equation}
\label{eq2}
\Delta_H \mathbf{u}(t)= f\left(t,\mathbf{u}(t)\right),
\quad \mathbf{u}(t_0)=\mathbf{u}_0,
\end{equation}
where $f\in C_{rd}[\mathbb{T} \times \mathbb{R}_\mathcal{F}^n,
\mathbb{R}_\mathcal{F}^n]$ and $\Delta_H u(t)$ denotes the
$\Delta_H$-derivative of $\mathbf{u}$  at $t\in \mathbb{T}$.

\begin{definition}
Let $V\in C_{rd}[\mathbb{T}\times \mathbb{R}_\mathcal{F}^n,\mathbb{R}_+]$.
We say that $D^+_\Delta V(t,\mathbf{u}(t))\in \mathbb{R}$
is the upper $\Delta$-Dini derivative of $V(t,\mathbf{u}(t))$
with respect to \eqref{eq2}, if for any given $\varepsilon>0$ there exists
a right neighborhood $U_\varepsilon$ of $t\in \mathbb{T}$ such that
$$
\frac{1}{\mu^*(t,s)}\big[V\big( \sigma(t),\mathbf{u}(\sigma(t))\big)
-V\big( s,\mathbf{u}(\sigma(t))\ominus_{gH}\mu^*(t,s) f(t,\mathbf{u}(t))\big)\big]
< D^+_\Delta V(t,\mathbf{u}(t))+\varepsilon
$$
for each $s\in U_\varepsilon$ and $s>t$,
where $\mathbf{u}(t)$ is solution of the dynamic system \eqref{eq2}.
\end{definition}

\begin{theorem}
Assume that $V\in C_{rd}[\mathbb{T}\times \mathbb{R}_\mathcal{F}^n,\mathbb{R}_+]$.
If $t\in \mathbb{T}$ is right-scattered and $V(t,\mathbf{u}(t))$
is continuous at $t$, then
$$
D^+_\Delta V(t,\mathbf{u}(t))
=\frac{V(\sigma(t),\mathbf{u}(\sigma(t)))-V(t,\mathbf{u}(t))}{\mu(t)}.
$$
\end{theorem}

\begin{proof}
Because of continuity of $V$ at $t$,
\begin{gather*}
\lim_{s \rightarrow t^+}\frac{V\big( \sigma(t),\mathbf{u}(\sigma(t))\big)
-V\big( s,\mathbf{u}(\sigma(t))\ominus_{gH}\mu^*(t,s)
f(t,\mathbf{u}(t))\big)}{\sigma(t)-s}\\
=\frac{V\big( \sigma(t),\mathbf{u}(\sigma(t))\big)
-V\big(t,\mathbf{u}(\sigma(t))\ominus_{gH}\mu(t)
\Delta_{H}\mathbf{u}(t)\big)}{\mu(t)}.
\end{gather*}
Since $u$ is continuous at $t$ and $t$ is right-scattered,
it follows from item (i) of Theorem~\ref{th1} that
\begin{gather*}
\lim_{s \rightarrow t^+}\frac{V\big( \sigma(t),\mathbf{u}(\sigma(t))\big)
-V\big(s,\mathbf{u}(\sigma(t))\ominus_{gH}\mu^*(t,s)
f(t,\mathbf{u}(t))\big)}{\sigma(t)-s}\\
=\frac{V\big( \sigma(t),\mathbf{u}(\sigma(t))\big)
-V\big(t,\mathbf{u}(t)\big)}{\mu(t)}.
\end{gather*}
From this last equality, we get that for a given $\varepsilon >0$ there exists
a right neighborhood $U_\varepsilon \subseteq \mathbb{T}$ of $t$ such that
\begin{gather*}
\frac{1}{\mu^*(t,s)}\big[V\big( \sigma(t),\mathbf{u}(\sigma(t))\big)
-V\big( s,\mathbf{u}(\sigma(t))\ominus_{gH}f(t,\mathbf{u}(t))\big)\big]\\
-\frac{V\big( \sigma(t),\mathbf{u}(\sigma(t))\big)
-V\big(t,\mathbf{u}(t)\big)}{\mu(t)}<\varepsilon.
\end{gather*}
Consequently,
\begin{gather*}
\frac{1}{\mu^*(t,s)}\big[V\big( \sigma(t),\mathbf{u}(\sigma(t))\big)
-V\big( s,\mathbf{u}(\sigma(t))\ominus_{gH}f(t,\mathbf{u}(t))\big)\big]\\
<\frac{V\big( \sigma(t),\mathbf{u}(\sigma(t))\big)
-V\big(t,\mathbf{u}(t)\big)}{\mu(t)}+\varepsilon .
\end{gather*}
This means that
$\left(V(\sigma(t),\mathbf{u}(\sigma(t)))-V(t,\mathbf{u}(t))\right)/\mu(t)$
is the upper $\Delta$-Dini derivative of $V(t,\mathbf{u}(t))$.
\end{proof}

\begin{theorem}
\label{th2}
Assume that
\begin{enumerate}[(i)]
\item $V\in C_{rd}[\mathbb{T}\times \mathbb{R}_\mathcal{F}^n,\mathbb{R}_+]$,
$V(t,\mathbf{u})$ is locally Lipshitzian in $\mathbf{u}$ for each right-dense
$t\in \mathbb{T}$ (i.e., $\exists$ $L>0$ such that
$\left|V(t,\mathbf{u_1})-V(t,\mathbf{u_2})\right|
<L \, D_\infty(\mathbf{u_1},\mathbf{u_2})$ for all
$\mathbf{u_1}, \mathbf{u_2}\in \mathbb{R}_\mathcal{F}^n$);

\item $D^+_\Delta V\left(t,\mathbf{u}(t)\right)
\leq g\left(t,V(t,\mathbf{u}(t))\right)$,
where $g\in C_{rd}[\mathbb{T}\times\mathbb{R}_+,\mathbb{R}]$
and $g(t,r)\mu(t)+r$ is nondecreasing in $r$ for each $t\in\mathbb{T}$;

\item $r(t)=r(t,t_0,r_0)$ is the maximal solution of
$r^\Delta(t) =g\left(t,r(t)\right)$, $r(t_0)=r_0\geq 0$,
existing on $\mathbb{T}$.
\end{enumerate}
Then, $V(t_0,\mathbf{u_0})\leq r_0$ implies that
$V(t,\mathbf{u}(t))\leq r(t,t_0,r_0)$ for all $t\in \mathbb{T}$, $t\geq t_0$.
\end{theorem}

\begin{proof}
The proof is similar to the one of \cite[Theorem 3.1.1]{laksh2}.
\end{proof}


\subsection{A comparison theorem}

We now establish a comparison result (Theorem~\ref{th3}),
which is useful to prove practical stability (Theorem~\ref{thm:ps}).
Because of the local nature of the solution $\mathbf{u}(t)$,
stability properties are proved with respect
to one condition on the Lyapunov-like function
$V(t,\mathbf{u}(t))$ that is defined only on
$\mathbb{T}\times S(\rho)$,
\begin{equation}
\label{eq:S}
S(\rho)=\left\{\mathbf{u}\in \mathbb{R}_\mathcal{F}^n :
D_\infty(\mathbf{u},\widetilde{0})<\rho\right\}.
\end{equation}

Let $\mathbb{T}$ be a nonnegative unbounded time scale with
$t_k \in \mathbb{T}$, $k \in \mathbb{N}_0$, satisfying
$0\leq t_0 < t_1 < t_2 <\cdots < t_k < \cdots$
and $t_k \rightarrow \infty$ as $k\rightarrow \infty$.
We consider the following hybrid fuzzy dynamic system:
\begin{equation}
\label{eq3}
\begin{gathered}
\Delta _H \mathbf{u}(t)=f\left(t,\mathbf{u}(t),\lambda(t,\mathbf{u}(t))\right),
\quad t \ge t_0, \quad t \in \mathbb{T},\\
\mathbf{u}(t_0)=\mathbf{u_0}\in S(\rho),
\end{gathered}
\end{equation}
where $f\in C_{rd}[\mathbb{T} \times S(\rho) \times \mathbb{R}_\mathcal{F}^n,
\mathbb{R}_\mathcal{F}^n]$, $S(\rho)$ is given by \eqref{eq:S}, and
$\lambda : \mathbb{T} \times \mathbb{R}_\mathcal{F}^n
\rightarrow \mathbb{R}_\mathcal{F}^n$
is a piecewise constant function defined by
$\lambda(t,\mathbf{u}(t)) = \lambda_k(t_k,\mathbf{u}(t_k))$
for $t \in [t_k, t_{k+1}]$, $k = 0, 1, 2, \ldots$.
If a solution $\mathbf{u}(t)$ to system \eqref{eq3} exists,
then it can be written as a piecewise function:
\begin{equation}
\label{eq10}
\mathbf{u}(t)=\mathbf{u}(t,t_0,\mathbf{u_0}) = \mathbf{u_k}(t),
\quad t_k \leq t \leq t_{k+1},
\quad k = 0, 1, 2, \ldots,
\end{equation}
with $\mathbf{u_k}(t)=\mathbf{u_k}(t, t_k, \mathbf{u_k})$
being the solution of $\Delta_H \mathbf{u}(t)
=f\left(t,\mathbf{u}(t),\lambda_k\left(t_k,\mathbf{u_k}\right)\right)$,
$\mathbf{u_k} = \mathbf{u}(t_k) \in S(\rho)$,
$t_k\leq t \leq t_{k+1}$, $k \in \mathbb{N}_0$.
We also consider the \emph{scalar comparison delta hybrid dynamic system}
\begin{equation}
\label{eq4}
\begin{gathered}
r^\Delta(t)=g\left(t,r(t),\psi(r(t))\right),
\quad  t \ge t_0, \quad t \in \mathbb{T},\\
r(t_0)=r_0\in \mathbb{R}_+,
\end{gathered}
\end{equation}
where $g\in C_{rd}[\mathbb{T} \times \mathbb{R}_+ \times \mathbb{R}_+, \mathbb{R}]$
and $\psi : \mathbb{R}_+  \rightarrow \mathbb{R}_+$ is a piecewise constant function:
$\psi(r(t)) = \psi_k(r(t_k))$, $t\in[t_k,t_{k+1}]$, $k=0,1,2,\ldots$.
Note that a piecewise function
\begin{equation}
\label{eq:rt}
r(t)= r(t,t_0,r_0) = r_k(t),
\quad t_k \leq t \leq t_{k+1},
\quad k = 0, 1, 2, \ldots
\end{equation}
is a maximal solution of \eqref{eq4} if and only if
$r_k(t) = r(t,t_k,r_k)$ is maximal solution of
$$
r^\Delta(t) =g\left(t,r(t),\psi_k(r_k)\right),
\quad r_k = r(t_k),
\quad t_k \leq t \leq t_{k+1},
\quad k = 0, 1, 2, \ldots.
$$
Now, we state and prove a comparison theorem with respect to a Lyapunov-like
function $V$. Here, the Lyapunov-like function $V$ serves as a vehicle
to transform the $\Delta$-Hukuhara hybrid fuzzy dynamic system \eqref{eq3}
into a scalar comparison delta hybrid dynamic system \eqref{eq4}, being enough
to consider the stability properties of the simpler comparison system \eqref{eq4}.

\pagebreak

\begin{theorem}
\label{th3}
Let $V\in C_{rd}[\mathbb{T}\times S(\rho) ,\mathbb{R}_+]$ be such that:
\begin{enumerate}[(i)]
\item $V(t,\mathbf{u})$ is locally Lipshitzian in
$\mathbf{u}$ for each right-dense $t\in \mathbb{T}$,
$t_k\leq t \leq t_{k+1}$, $k=0,1,2,\ldots$;

\item $D^+_\Delta V(t,\mathbf{u}(t))\leq g(t,V(t,\mathbf{u}(t)),
\psi_k(V(t_k,\mathbf{u_k})))$, where
$g\in C_{rd}[\mathbb{T}\times\mathbb{R}_+\times \mathbb{R}_+, \mathbb{R}]$,
$\psi_k : \mathbb{R}_+ \rightarrow \mathbb{R}_+$, $g(t,r,v)\mu(t)+r$ is
nondecreasing in $r$ for each $(t,v)$, and $\psi_k(v)$ and $g(t,r,v)$
are nondecreasing in $v$.
\end{enumerate}
Moreover, let
\begin{enumerate}
\item[(iii)] $r(t)$ given by \eqref{eq:rt} be the maximal solution of the
scalar comparison hybrid dynamic system \eqref{eq4},
which we assume to exist for each $t \geq t_0$, $t\in \mathbb{T}$.
\end{enumerate}
If $\mathbf{u}(t)$ given by \eqref{eq10} is a solution
of the hybrid fuzzy dynamic system \eqref{eq3} with
$V(t_0,\mathbf{u_0})\leq r_0$, then
\begin{equation}
\label{eq5}
V(t,\mathbf{u}(t))\leq r(t) \ \text{ for all }\
t\geq t_0, \quad t\in \mathbb{T}.
\end{equation}
\end{theorem}

\begin{remark}
Note that by \eqref{eq10}, inequality \eqref{eq5} implies that
$V(t,\mathbf{u_k}(t)) \leq r_k(t)$,
$t_k \leq t \leq t_{k+1}$, for all $k \in \mathbb{N}_0$.
\end{remark}

\begin{proof}
Let $\mathbf{u}(t)$ be a solution of \eqref{eq3}
on $[t_0,\infty)\cap\mathbb{T}$.
Set $m(t) = V(t, \mathbf{u}(t))$. Then,
\begin{equation*}
\frac{m(\sigma(t))-m(t)}{\mu(t)}
=\frac{V(\sigma(t),\mathbf{u}(\sigma(t)))
-V(t,\mathbf{u}(t))}{\mu(t)}.
\end{equation*}
Moreover, by condition (ii), the inequality
\begin{equation*}
D^+_\Delta m(t) \leq  g(t,V(t,\mathbf{u}(t)),
\psi_k(V(t_k,\mathbf{u_k})))=g\left(t,m(t),\psi_k(m_k)\right)
\end{equation*}
is obtained for $t_k \leq t \leq t_{k+1}$, where
$m_k=V(t_k,\mathbf{u_k})$. First, consider $t\in [t_0,t_1]\cap \mathbb{T}$.
Since $m(t_0)=V(t_0,\mathbf{u_0})\leq r_0$, by Theorem~\ref{th2} we conclude that
\begin{equation}
\label{eq6}
V(t,\mathbf{u_0}(t))\leq r_0(t,t_0,r_0),
\quad t_0 \leq t \leq t_1,
\end{equation}
where $\mathbf{u_0}(t) = \mathbf{u_0}(t,t_0,\mathbf{u_0})$ is the solution of
\begin{equation*}
\Delta _H \mathbf{u_0}(t)
=f(t,\mathbf{u_0}(t),\lambda_0(t_0,\mathbf{u_0})),
\quad \mathbf{u_0}(t_0)=\mathbf{u_0},
\quad t_0 \leq t \leq t_1,
\end{equation*}
and $r_0(t) = r_0(t,t_0,r_0)$ is the maximal solution of
\begin{equation*}
r_0^\Delta(t) =g(t,r_0(t),\psi_0(r_0)),
\quad r_0(t_0)=r_0,
\quad t_0 \leq t \leq t_1.
\end{equation*}
Now, choose $\mathbf{u_1}=\mathbf{u_0}(t_1)$. Then,
\begin{equation*}
D^+_\Delta m(t) \leq g\left(t,m(t),\psi_1(m_1)\right),
\quad t_1 \leq t \leq t_2,
\end{equation*}
where $m_1=m(t_1)=V(t_1,\mathbf{u_0}(t_1))$.
On the other hand, the inequality \eqref{eq6} gives us
\begin{equation*}
 V(t_1,\mathbf{u_0}(t_1))\leq r_0(t_1,t_0,r_0).
\end{equation*}
Set $r_0(t_1,t_0,r_0)=r_1$. Due to the monotone property
of $\psi_1$ and $g(t,r,v)$ in $v$,
\begin{equation*}
 D^+_\Delta m(t) \leq g(t,m(t),\psi_1(r_1))
\end{equation*}
and
\begin{equation*}
m(t) \leq r_1,
\quad t\in [t_1,t_2].
 \end{equation*}
Similarly,
\begin{equation*}
V(t,\mathbf{u_1}(t))\leq r_1(t,t_1 ,r_1),
\quad t \in [t_1, t_2],
\end{equation*}
is established, where
$\mathbf{u_1}(t) = \mathbf{u_1}(t,t_1,\mathbf{u_1})$ is the solution of
\begin{equation*}
\Delta _H \mathbf{u_1}(t)=f(t,\mathbf{u_1}(t),\lambda_1(t_1,\mathbf{u_1})),
\quad \mathbf{u_1}(t_1)=\mathbf{u_1},
\quad t_1 \leq t \leq t_2,
\end{equation*}
and $r_1(t) = r_1(t,t_1,r_1)$ is the maximal solution of
\begin{equation*}
r_1^\Delta(t) =g(t,r_1(t),\psi_1(r_1)),
\quad r_1(t_1)=r_1,
\quad t_1 \leq t \leq t_2.
\end{equation*}
By repeating the process and using the special choice
\begin{equation*}
\mathbf{u_k}=\mathbf{u_{k-1}}(t_k),
\quad k=1,2,\ldots,
\end{equation*}
we have
\begin{equation*}
V(t,\mathbf{u_k}(t))\leq r_k(t,t_k ,r_k),
\quad t_k \leq t \leq t_{k+1},
\end{equation*}
where $\mathbf{u_k}(t) = \mathbf{u_k}(t,t_k,\mathbf{u_k})$ is the solution of
\begin{equation*}
\Delta _H \mathbf{u_k}(t)
=f(t,\mathbf{u_k}(t),\lambda_k(t_k,\mathbf{u_k})),
\quad \mathbf{u_k}(t_k)=\mathbf{u_k},
\quad t_k \leq t \leq t_{k+1},
\end{equation*}
and $r_k(t) = r_k(t,t_k,r_k)$ is the maximal solution of
\begin{equation*}
r_k^\Delta(t) = g\left(t,r_k(t),\psi_k(r_k)\right),
\quad r_k(t_k)=r_k,
\quad t_k \leq t \leq t_{k+1},
\end{equation*}
with $r_{k-1}(t_k) = r_{k-1}(t_k,t_{k-1},r_{k-1})=r_k$ for each $k=2,3,\ldots$
By means of these inequalities, we end with the desired result.
\end{proof}


\subsection{A criterium for practical stability}

Now, we provide a sufficient condition for practical stability
of the $\Delta$-Hukuhara hybrid fuzzy dynamic system \eqref{eq3}.
The definitions of practical stability for \eqref{eq3} are similar
to the corresponding notions of practical stability for
the scalar comparison delta hybrid dynamic system \eqref{eq4} \cite{laksh3}.

\begin{definition}
\label{def2}
The hybrid fuzzy dynamic system \eqref{eq3} is
\begin{itemize}
\item \textit{practically stable} if, for given
$\lambda, A \in \mathbb{R}_+$ with $0 < \lambda < A$,
\begin{equation}
\label{(1)}
D_\infty(\mathbf{u_0},\widetilde{0})<\lambda \Rightarrow
D_\infty(\mathbf{u}(t),\widetilde{0})<A
\ \forall \ t \ge t_0,
\end{equation}
$t \in \mathbb{T}$, where $\mathbf{u}(t)$
is any solution \eqref{eq10} of \eqref{eq3};

\item \textit{practically quasi-stable} if,
for given $(\lambda,B,T_0)>0$
such that $t_0+T_0 \in \mathbb{T}$,
\begin{equation}
\label{(2)}
D_\infty(\mathbf{u_0},\widetilde{0})<\lambda \Rightarrow
D_\infty(\mathbf{u}(t),\widetilde{0})<B
\ \forall \ t\geq t_0+T_0,
\end{equation}
$t \in \mathbb{T}$, where $\mathbf{u}(t)$
is any solution \eqref{eq10} of \eqref{eq3};

\item \textit{strongly practically stable}
if both \eqref{(1)} and \eqref{(2)} hold;

\item \textit{practically asymptotically stable} if \eqref{(1)} holds and for
any $\varepsilon>0$ there exists $T_0>0$ such that $t_0+T_0 \in \mathbb{T}$ and
$D_\infty(\mathbf{u_0},\widetilde{0})<\lambda$ implies
$D_\infty(\mathbf{u}(t),\widetilde{0})<A$ for all $t\geq t_0+T_0$
in the time scale $\mathbb{T}$, where
$\mathbf{u}(t)$ is any solution \eqref{eq10} of \eqref{eq3}.
\end{itemize}
\end{definition}

\begin{theorem}
\label{thm:ps}
Assume that $0 < \lambda < A$. Let
\begin{enumerate}[(i)]
\item $V\in C_{rd}[\mathbb{T}\times S(A),\mathbb{R}_+]$, $V(t,\mathbf{u})$
be locally Lipshitzian in $u$ for each right-dense $t\in \mathbb{T}$,
$t_k\leq t \leq t_{k+1} , k=0,1,2,\ldots$;

\item $D^+_\Delta V(t,\mathbf{u}(t))\leq g(t,V(t,\mathbf{u}(t)),\psi_k(V(t_k,u_k)))$,
where $g\in C_{rd}[\mathbb{T}\times\mathbb{R}_+\times \mathbb{R}_+,\mathbb{R}]$,
$\psi_k : \mathbb{R}_+ \rightarrow \mathbb{R}_+$, $g(t,r,v)\mu(t)+r$ is nondecreasing
in $r$ for each $(t,v)$, and $\psi_k(v)$ and $g(t,r,v)$ are nondecreasing in $v$;

\item $V(t,\mathbf{u})$ satisfies
$$
b(D_\infty(\mathbf{u},\widetilde{0}))
\leq V(t,\mathbf{u})
\leq a(D_\infty(\mathbf{u},\widetilde{0})),
\quad (t,\mathbf{u})\in \mathbb{T}\times S(A),
$$
where $a,b \in \mathcal{K} := \left\{f\in C[\mathbb{R}_+,
\mathbb{R}_+] : f(x) \text{ is increasing and }
f(0)=0\right\}$, and $a(\lambda)<b(A)$.
\end{enumerate}
Then, any practical stability property (practically stable,
practically quasi-stable, strongly practically stable or
practically asymptotically stable) of the scalar comparison
delta hybrid dynamic system \eqref{eq4}, imply the corresponding practical
stability property of the $\Delta$-Hukuhara hybrid fuzzy dynamic system
\eqref{eq3} (see Definition~\ref{def2}).
\end{theorem}

\begin{proof}
First, suppose that the scalar comparison hybrid dynamic system \eqref{eq4}
is practically stable. It follows that for given
$a(\lambda),b(A) \in \mathbb{R}_+$ we have
\begin{equation}
\label{eq7}
0 \leq r_0<a(\lambda) \Longrightarrow  r(t,t_0,r_0)<b(A),
\quad t\geq t_0,
\quad t \in \mathbb{T},
\end{equation}
where $r(t) = r(t, t_0, r_0)$ is a solution of \eqref{eq4}. Consider
$D_\infty(\mathbf{u_0},\widetilde{0})<\lambda$. Now, we assert that
$D_\infty(\mathbf{u}(t),\widetilde{0})<A$, $t\geq t_0$, $t\in \mathbb{T}$,
where $\mathbf{u}(t) = \mathbf{u}(t, t_0, \mathbf{u_0})$
is a solution of \eqref{eq3}. Assume, to the contrary, that there exists
a $t_1\geq t_0$ and a solution $\mathbf{u}(t) = \mathbf{u}(t, t_0, \mathbf{u_0})$
with  $D_\infty(\mathbf{u_0},\widetilde{0})<\lambda$ such that
$$
D_\infty(\mathbf{u}(t_1),\widetilde{0})\geq A \text{ and }
D_\infty(\mathbf{u}(t),\widetilde{0})<A,
\quad t_0\leq t< t_1.
$$
This implies that
$V(t_1,\mathbf{u}(t_1))
\geq b(D_\infty(\mathbf{u_0},\widetilde{0}))
\geq b(A)$. Choose $V (t_0, \mathbf{u_0}) = r_0$ and make the special choice of
$\mathbf{u_k}$ and $r_k$ as in the proof of Theorem~\ref{th3}, to arrive at
$$
V(t,\mathbf{u}(t))\leq r(t,t_0,V (t_0, \mathbf{u_0})),
\quad t_0\leq t < t_1.
$$
Thus,
\begin{equation*}
\begin{split}
b(A)&\leq b(D_\infty(\mathbf{u_0},\widetilde{0}))
\leq V(t_1,\mathbf{u}(t_1)) \leq r(t_1,t_0,V (t_0, \mathbf{u_0}))\\
&\leq r(t_1,t_0,a(D_\infty(\mathbf{u_0},\widetilde{0})))
\leq r(t_1,t_0,a(\lambda))<b(A).
\end{split}
\end{equation*}
This contradiction proves that the hybrid fuzzy dynamic system \eqref{eq3}
is practically stable. Note that the last inequality arises from relation
\eqref{eq7}. Secondly, we prove that the hybrid fuzzy dynamic system \eqref{eq3}
is practically quasi-stable, if the scalar comparison hybrid dynamic system
\eqref{eq4} is practically quasi-stable. Due to quasi-stability of \eqref{eq4},
we deduce that
$$
0 \leq r_0<a(\lambda)
\Longrightarrow r(t,t_0,r_0)<b(B),
\quad t\geq t_0+T_0,
\quad t, t_0+T_0 \in \mathbb{T},
$$
where $r(t) = r(t, t_0, r_0)$ is a solution of \eqref{eq4}.
Suppose that $D_\infty(\mathbf{u_0},\widetilde{0})<\lambda$.
From Theorem~\ref{th3},
$V(t,\mathbf{u}(t))\leq r(t,t_0,V (t_0, \mathbf{u_0}))$,
$t \geq t_0$. Set $V (t_0, \mathbf{u_0}) = r_0$.
Then,
\begin{equation*}
\begin{split}
b\left(D_\infty(\mathbf{u}(t),\widetilde{0})\right)
&\leq V(t,\mathbf{u}(t))
\leq r(t,t_0,V (t_0, \mathbf{u_0}))\\
&\leq r(t,t_0,a(D_\infty(\mathbf{u_0},\widetilde{0})))
\leq r\left(t,t_0,a(\lambda)\right)< b(B)
\end{split}
\end{equation*}
for all $ t\geq t_0+T_0$.
Because $b$ is an increasing function,
$D_\infty(\mathbf{u}(t),\widetilde{0})<B$ and, therefore,
we have practical quasi-stability of \eqref{eq3}.
The strongly practically stability of \eqref{eq3} is obvious;
practical asymptotic stability of \eqref{eq3} is proved similarly.
\end{proof}

We illustrate the applicability of Theorem~\ref{thm:ps} with an example. Recall
that a function $p: \mathbb{T} \rightarrow \mathbb{R}$ is called regressive provided
$1+ \mu(t)p(t)\neq 0$ for all $t \in \mathbb{T}^{\kappa}$. For given regressive functions
$p$ and $q$, the ``circle plus'' and ``circle minus'' operations are defined,
respectively, by $p\oplus_{r} q=p+q+\mu p q$,
$p\ominus_{r} q = \frac{p-q}{1+\mu q}$,
and $\ominus_{r} p = 0 \ominus_{r} p$ (see, e.g., \cite{boh}).
It is easy to check that $p\oplus_{r} (\ominus_{r} q)= p\ominus_{r} q$,
$\ominus_{r} (\ominus_{r} p)= p$, $p\ominus_{r} q = \ominus_{r} (q \ominus_{r} p)$,
and $p\ominus_{r} p =0$. Note that $\ominus_{r} 1 = 0 \ominus_{r} 1 = \frac{-1}{1+\mu}$.
For convenience, we denote $\ominus_{r} 1$ by $\ominus_{r}$.

Now we are ready to illustrate our approach
with a simple example of a hybrid fuzzy
dynamic system on time scales.

\begin{example}
Let us consider the following hybrid fuzzy dynamic system
on the time scale $\mathbb{T}=\mathbb{N}_0$:
\begin{equation}
\label{exm}
\vartriangle_H \mathbf{u}(t) = \ominus_{r} \mathbf{u}(t)
\oplus \eta(t)\lambda_{k}(\mathbf{u}_{k}),
\quad t\in[\tau_k,\tau_{k+1}],
\end{equation}
$$
\mathbf{u}(\tau_k)=\mathbf{u_k}\in S(\rho),
\quad k=0,1,2,\ldots,
$$
where $\eta(t)=\frac{1}{1+\mu(t)}$ and
$$
\lambda_{k}(\tau) =
\begin{cases}
\widetilde{0} & \text{ if } k=0,\\
\tau & \text{ if } k \in \{1,2,\ldots\}.
\end{cases}
$$
Note that all points $t$ of the time scale $\mathbb{T}$ are right-scattered.
Let us choose $V(t,\mathbf{u}(t))
=D_\infty\left(\mathbf{u}(t), \widetilde{0}\right)$ for all $t \in\mathbb{T}$.
If $\mathbf{u}(t) = u\left(t, t_0, \mathbf{u}_0\right)$ is a solution
of \eqref{exm} corresponding to the initial value
$\mathbf{u}(t_0) = \mathbf{u}(\tau_0) = \mathbf{u}(0) = \mathbf{u}_0$, then we have
$$
D^+_\vartriangle V(t,\mathbf{u}(t))
=\frac{V(\sigma(t),\mathbf{u}(\sigma(t)))
-V(t,\mathbf{u}(t))}{\mu(t)}
=\frac{D_\infty\left(\mathbf{u}(\sigma(t)), \widetilde{0}\right)
-D_\infty\left(\mathbf{u}(t), \widetilde{0}\right)}{\mu(t)}.
$$
Now, let us take $g(t,w, \psi(w))=\frac{w+w_k}{1+\mu(t)}$.
Since $t$ is right-scattered, then
\begin{equation*}
\begin{split}
D^+_\vartriangle V\left(t,\mathbf{u}(t)\right)
&=\frac{D_\infty\left(\mathbf{u}(\sigma(t)), \widetilde{0}\right)
-D_\infty\left(\mathbf{u}(t), \widetilde{0}\right)}{\mu(t)}\\
&\leq \frac{D_\infty(\mathbf{u}(\sigma(t))
\ominus_{gH} \mathbf{u}(t), \widetilde{0})}{\mu(t)}
= D_\infty\left(\vartriangle_H \mathbf{u}(t), \widetilde{0}\right)\\
&=D_\infty\left(\ominus_{r} \mathbf{u}(t)
\oplus \eta(t)\lambda_{k}(\mathbf{u}_{k}), \widetilde{0}\right)
= D_\infty\left(\frac{\mathbf{u}_{k}
\oplus (-1)\cdot\mathbf{u}(t)}{1+\mu(t)},\widetilde{0}\right)\\
&\leq g\left(t,D_\infty\left(\mathbf{u}(t),\widetilde{0}\right),
\psi\left(D_\infty\left(\mathbf{u}(t),\widetilde{0}\right)\right)\right).
\end{split}
\end{equation*}
It follows from Theorem~\ref{thm:ps} that any practical stability property
of the solution of the system \eqref{eq4} with $g(t,w, \psi(w))
=\frac{1}{1+\mu(t)}(w+w_k)$ and $r(t_0)=D_\infty(\mathbf{u}_0, \widetilde{0})$
implies the corresponding stability property of the solution to \eqref{exm}.
\end{example}


\section{A Remark on Some Previous Results}
\label{sec:rem:pr}

In \cite{samban}, Sambandham considers the delta-derivative $f^\Delta$
for a function $f: \mathbb{T}\rightarrow X$, where $X$ is a Banach space.
He investigates the following hybrid fuzzy system on time scales:
\begin{equation}
\label{eq:Samb}
\begin{gathered}
\mathbf{u}^\Delta = f(t,\mathbf{u},\lambda_k(t_k,\mathbf{u_k})),
\quad t\in [t_k,t_{k+1}],\\
\mathbf{u}(t_k)=\mathbf{u_k}\in S(\rho),
\quad k=0,1,2,\ldots,
\end{gathered}
\end{equation}
where $S(\rho)=\left\{\mathbf{u}\in \mathbb{R}_\mathcal{F}^n :
D_\infty(\mathbf{u},\widetilde{0})<\rho\right\}$. Unfortunately,
in \cite{samban} it is mistakenly assumed that the space of all fuzzy subsets in
$\mathbb{R}^n$, $\mathbb{R}_\mathcal{F}^n$, is a Banach space.
However, $\mathbb{R}_\mathcal{F}^n$ is just a complete metric space and not
a Banach space due to the fact that $\mathbb{R}_\mathcal{F}^n$ is not a normed
space \cite{bede2}. Here we note that such inconsistency
in \cite{samban} is easily overcome. For that we make use of the well-known
embedding theorem (see, e.g., \cite{bede2,bede,puri}), to embed the space
$\mathbb{R}_\mathcal{F}^n$ into a Banach space.

\begin{theorem}[Embedding theorem \cite{puri}]
\label{thm:emb}
There exists a real Banach space $X$ such that $\mathbb{R}_\mathcal{F}^n$
can be embedded as a closed convex cone $C$ with vertex $0$ in $X$.
Furthermore, the embedding $j$ is an isometry (i.e., $j$ preserves distance).
\end{theorem}

Let us denote by $\parallel \cdot \parallel$ the function
$\parallel \mathbf{u} \parallel =D_\infty (\mathbf{u},\widetilde{0})$ defined
for $\mathbf{u} \in \mathbb{R}_\mathcal{F}^n$. The next lemma asserts that
$\parallel \cdot \parallel$  has properties similar to the properties of a norm
in the usual crisp sense, without being a norm. It is not a norm because
$\mathbb{R}_\mathcal{F}^n$ is not a linear space and, consequently,
$(\mathbb{R}_\mathcal{F}^n, \parallel \cdot \parallel)$ is not a normed space.

\begin{lemma}
\label{lemma:e}
Function $\parallel \cdot \parallel$ has the following properties:
\begin{enumerate}[(i)]
\item $\parallel \mathbf{u} \parallel = 0$
if and only if $\mathbf{u}=\widetilde{0}$;

\item $\parallel \lambda \cdot \mathbf{u} \parallel
=|\lambda| \cdot \parallel \mathbf{u} \parallel$
for all $\mathbf{u} \in \mathbb{R}_\mathcal{F}^n$ and $\lambda \in \mathbb{R}$;

\item $\parallel \mathbf{u} \oplus \mathbf{v} \parallel
\leq \parallel \mathbf{u} \parallel + \parallel \mathbf{v} \parallel $
for all $\mathbf{u},\mathbf{v} \in \mathbb{R}_\mathcal{F}^n$.
\end{enumerate}
\end{lemma}

\begin{proof}
\begin{enumerate}[(i)]

\item By definition of $\parallel \cdot \parallel$, we have that $\parallel
\mathbf{u} \parallel = 0$ if and only if $D_\infty (\mathbf{u},\widetilde{0})=0$.
Since $(\mathbb{R}_\mathcal{F}^n,D_\infty)$
is a metric space, $\mathbf{u}=\widetilde{0}$.

\item For each $\mathbf{u}\in \mathbb{R}_\mathcal{F}^n$
and $\lambda \in \mathbb{R}$ we have that
$$
\parallel \lambda \cdot \mathbf{u}\parallel
=D_\infty(\lambda \cdot \mathbf{u},\widetilde{0})
=D_\infty(\lambda \cdot \mathbf{u},\lambda \cdot \widetilde{0})
=|\lambda| D_\infty(\mathbf{u},\widetilde{0})
=|\lambda| \parallel \mathbf{u} \parallel
$$
because of item (ii) of Proposition~\ref{pro1}.

\item The intended inequality follows from item (iii) of Proposition~\ref{pro1}:
\begin{equation*}
\begin{split}
\parallel \mathbf{u} \oplus \mathbf{v} \parallel
&=D_\infty(\mathbf{u} \oplus \mathbf{v},\widetilde{0})
=D_\infty (\mathbf{u} \oplus \mathbf{v},\widetilde{0}\oplus \widetilde{0})\\
&\leq D_\infty (\mathbf{u},\widetilde{0})+D_\infty(\mathbf{v},\widetilde{0})
=\parallel \mathbf{u}  \parallel +\parallel \mathbf{v} \parallel.
\end{split}
\end{equation*}
\end{enumerate}
The proof is complete.
\end{proof}

From Lemma~\ref{lemma:e}, we deduce that all sufficient conditions presented
in \cite{samban} hold true in the Banach space $X$ asserted by
Theorem~\ref{thm:emb}. Indeed, it is enough to replace
$\mathbb{R}_\mathcal{F}^n$ by $j(\mathbb{R}_\mathcal{F}^n)$ in \cite{samban}
to conclude with the validity of the sufficient conditions of \cite{samban}
for the practical stability of the hybrid fuzzy system on time scales
\eqref{eq:Samb} in $j(\mathbb{R}_\mathcal{F}^n)$. Because $j$ is invertible,
one can then extend the results to $\mathbb{R}_\mathcal{F}^n$.


\section{Conclusion}
\label{sec:conc}

We investigated hybrid fuzzy systems on time scales with two
outstanding purposes: to establish practical stability of hybrid fuzzy systems
on time scales in the Lyapunov sense, based on the delta-Hukuhara derivative;
to improve and state a clarification of the results of \cite{samban}.
Furthermore, a comparison theorem was discussed, which is useful to prove
the practical stability criterion. A smart example of a hybrid fuzzy
dynamic system on time scales was also stated and discussed,
illustrating the main results of the paper.


\pagebreak

\noindent \textbf{Acknowledgements.}
This work is part of last author's PhD project.
It was partially supported by Damghan University, Iran;
and CIDMA--FCT, Portugal, within project UID/MAT/04106/2013.
The authors are grateful to Professor Dr. Milan Merkle
and three anonymous reviewers for comments and suggestions.


\vspace{1cc}


{\small

\noindent Omid Solaymani Fard \hfill (Received July 22, 2015)\newline
\noindent School of Mathematics and Computer Science\hfill (Revised February 10, 2016)\newline
\noindent Damghan University, Damghan, Iran\\
\noindent E-mails: {\tt osfard@du.ac.ir, omidsfard@gmail.com}

\medskip

\noindent Delfim F. M. Torres\newline
\noindent Center for Research and Development in Mathematics and Applications (CIDMA)\newline
\noindent Department of Mathematics, University of Aveiro, 3810--193 Aveiro, Portugal\\
\noindent E-mail: {\tt delfim@ua.pt}

\medskip

\noindent Mohadeseh Ramezan Zadeh\newline
\noindent Department of Mathematics, Damghan University, Damghan, Iran\\
\noindent E-mail: {\tt ramezanzadeh2@ymail.com}

}


\end{document}